\documentclass[11pt]{amsart}
\usepackage{amsmath,amsthm,amsfonts,amssymb,verbatim,eucal}
\usepackage{mathtools}

\numberwithin{equation}{section}
 
\newtheorem{thm}[equation]{Theorem} 

\newtheorem{lemma}[equation]{Lemma} 

\newtheorem{example}[equation]{Example}
\newtheorem{remark}[equation]{Remark}

\newcommand{\ep}{\epsilon}

\newcommand{\DOT}{\setlength{\unitlength}{1pt}\begin{picture}(2.5,2)
               (1,1)\put(2.5,2.5){\circle*{2}}\end{picture}}
\newcommand{\bu}{\DOT}

\newcommand{\Hom}{\mbox{\rm Hom\,}}

\newcommand{\ot}{\otimes}

\newcommand{\del}{\partial}

\newcommand{\Wedge}{\textstyle\bigwedge}

\begin{document}
\begin{abstract}
The deformation theory of an algebra is controlled by the Gerstenhaber bracket, 
a Lie bracket on Hochschild cohomology.
We develop techniques for evaluating Gerstenhaber brackets
of semidirect product algebras
recording actions of finite groups
over fields of positive characteristic. 
The Hochschild cohomology and Gerstenhaber bracket
of these skew group algebras can be complicated
when the characteristic of the underlying field divides the group
order.  We show how to investigate Gerstenhaber brackets using
twisted product resolutions, which are often smaller
and more convenient than the cumbersome bar resolution
typically used.  These resolutions provide a concrete description of the Gerstenhaber bracket
suitable for exploring questions in deformation theory.
We demonstrate with the prototypical example of 
a graded Hecke algebra (rational Cherednik
algebra) in positive characteristic.
\end{abstract}

\title[Gerstenhaber brackets]
{Gerstenhaber brackets\\ for skew group algebras\\ in positive characteristic}

\date{May 22, 2019.}
\thanks{The first author was partially supported by Simons grant 429539.
The second author was partially supported by NSF grant DMS-1665286.
Corresponding author: Anne Shepler.}

\author{A.V.\ Shepler}
\address{Department of Mathematics, University of North Texas,
Denton, Texas 76203, USA}
\email{ashepler@unt.edu}
\author{S.\  Witherspoon}
\address{Department of Mathematics\\Texas A\&M University\\
College Station, Texas 77843, USA}\email{sjw@math.tamu.edu}

\maketitle

\section{Introduction}

The Hochschild cohomology space of an associative algebra
is a Gerstenhaber algebra under two binary operations, 
the cup product and the Gerstenhaber bracket.
The Gerstenhaber bracket is a Lie bracket controlling the deformation theory
of the algebra.  Historically, it has been more difficult to compute
than the cup product:
The bracket is defined in terms of the cumbersome bar resolution
and notoriously resists transfer 
to more convenient resolutions.  In general, we lack user-friendly
formulas giving the Gerstenhaber bracket
explicitly.

We consider the Hochschild cohomology of a skew group algebra 
(semidirect product algebra) arising from the 
action of a finite group $G$ on an algebra $S$.  
We work in the modular setting,
i.e., over a field $k$ of positive characteristic that may divide the group order $|G|$.
In this setting, the Hochschild cohomology of $S\rtimes G$
is complicated by the potentially onerous cohomology of $kG$, in contrast to the characteristic zero case where it is always trivial.

Computations of the Gerstenhaber bracket on $S\rtimes G$
directly using the bar resolution often yield little 
useful information---the bar resolution itself is too large and unwieldy. 
It can be a struggle even to  
describe adequately the Hochschild cohomology
using the bar resolution.
Thus one seeks a description of the Gerstenhaber
bracket in terms of smaller resolutions used to compute
Hochschild cohomology,
a description that is concrete and straightforward 
to apply in specific examples.

In this note, we consider the flexible {\em twisted product resolution}
of a skew group algebra: one chooses a
convenient resolution for $S$ and another for $G$
and then combines them to create a resolution of $S\rtimes G$.
We show how to apply new techniques from~\cite{NW1}
on Gerstenhaber brackets to
twisted product resolutions for skew group algebras from~\cite{ueber,twisted}.
This approach provides advantages over employing the often unmanageable but
traditional bar resolution.
We produce an explicit description
of the Gerstenhaber bracket that should prove user-friendly and 
we illustrate
with an example from deformation theory.
This quintessential example using a small
transvection group captures the difference between
the modular and nonmodular settings,
both in the theory of reflection groups and 
in the theory of graded Hecke algebras (and rational
Cherednik algebras, see~\cite{Norton}).

In Section~\ref{sec:NW1}, we recall the twisted product resolution
from~\cite{ueber,twisted} obtained by twisting a resolution of $S$
with one for $G$.
We recall methods of~\cite{NW1} analyzing Gerstenhaber brackets
in Section~\ref{sec:gb}
and show
how they apply to twisted product resolutions for skew group algebras. 
We illustrate these techniques by 
showing how to compute some Gerstenhaber brackets
concretely
for a small transvection group example from~\cite{ueber} in Section~\ref{sec:example}. 
Throughout, $k$ is a field of arbitrary characteristic
and $\ot = \ot_k$. 

%%%%%%%%%%%%%%%%%%%%%%%%%%%%%%%%%%%%%%%%%%%%%%%%%%%%%%%%%%%%%%%%%%

\section{Twisted product resolutions}\label{sec:NW1}
We recall the twisted product resolution from~\cite{ueber,twisted}.
Consider a finite group $G$ acting on a $k$-algebra $S$ by automorphisms. 
Let $A=S\rtimes G$ be the corresponding skew group algebra:
As a vector space, $S\rtimes G = S\ot kG$,
and we abbreviate the element $s\ot g$ by $sg$ ($s\in S$, $g\in G$)
when no confusion can arise.
Multiplication is defined by
\[
   (sg) \cdot (s'g') = s ({}^gs') \, gg'
\quad\text{ for all }
s,s'\in S \text{ and } g,g'\in G\, .
\]
The action of $g$ on $s'$ here is denoted by $ {}^g s'$.
We use the enveloping algebra $S^e=S\ot S^{op}$
of any algebra $S$
to express bimodule actions as left actions.

%%%%%%%%%%%%%%%%%%%%%%%%%
\subsection*{The twisted product resolution}
We consider projective resolutions
$$
\begin{aligned}
%\quad
\text{ (i) }\ \ 
C:\, \ldots \rightarrow & \ \, C_2 \rightarrow C_1 \rightarrow C_0 \rightarrow 0
\ \ \text{ of } kG \text{ as a $kG$-bimodule, and}\\
%\quad
\text{ (ii)  } \ \ 
D:\, \ldots \rightarrow & \ D_2 \rightarrow D_1 \rightarrow D_0 \rightarrow 0\
\ \ \text{ of } S \text{ as an $S$-bimodule.}
\end{aligned}
$$
We assume the resolution $C$ is $G$-graded, with compatible group action:
\begin{equation}
\label{G-grading}
g_1\big((C_i)_{g_2}\big)g_3 = (C_i)_{g_1g_2g_3}
\quad\text{ for all }\ g_1,g_2,g_3\in G
\ \text{and all degrees } i .
\end{equation}
We also assume $D_{\DOT}$ carries a compatible action of $G$:
Each $D_{i}$ is left $kG$-module 
with 
\begin{equation}
\label{D-compatible}
g\cdot (s\cdot d) = {}^gs\cdot (g\cdot d)
\quad\text{ for all }
g\in G,\ s\in S,\ d \in D
\end{equation}
and the differentials are $kG$-module homomorphisms. 
This ensures  $D_{\DOT}$ is 
 {\em compatible} with the twisting map $g \ot s \mapsto {}^gs \ot g$
given by the group action (see~\cite[Definition~2.17]{twisted}). 
This is the setting, for example, when $C_{\bu}$ is the bar or reduced bar resolution
of $kG$ and when $D_{\bu}$ is the Koszul resolution
of a Koszul algebra $S$
(see~\cite[Prop 2.20(ii)]{twisted}).

The {\em twisted product resolution} $X=C \ot^G D$ 
of the algebra $S\rtimes G$ is
the total complex of the double complex $C_{\DOT}\ot D_{\DOT}$,
$$
X = C \ot^G D
\quad\text{ where }\quad
X_n = \bigoplus_{i+j=n} C_i \ot D_j\, 
$$
with
each $X_n$ suffused with the additional structure
of a $(S\rtimes G)$-bimodule defined by
\[
   s'g' \cdot (c\ot d)\cdot sg
   = 
   g'cg \ \ot\ ({}^{(g'hg)^{-1}}s' )({}^{g^{-1}}(ds))
\ \text{ for }\ 
c\in C_h, d\in D, g,g', h\in G, s,s'\in S\, .
\]
The differential on $X$ is $\del_n = \sum_{i+j=n} (\del_i \ot 1)+ (-1)^i (1\ot \del_j )$ as usual.

With this action, $X$ is a resolution of $A=S\rtimes G$,
i.e., $X$ provides an
exact sequence of $A$-bimodules 
(see~\cite{twisted} or ~\cite[\S4]{quad}):
$$
\ldots\rightarrow X_2\rightarrow X_1 \rightarrow X_0 \rightarrow
A \rightarrow 0\, .
$$

When the $A$-bimodules $X_n$ are all projective as
$A^e$-modules, $X$ is also a projective resolution of $A$.
This occurs, for example, when $D$ is a  Koszul resolution
of a Koszul algebra and $C$ is the bar resolution of $kG$.
(See~\cite[Proposition 2.20(ii)]{twisted}.)

%%%%%%%%%%%%%%%%%%%%%%%%%%%%%%%%%%%%%
%%%%%%%%%%%%%%%%%%%%%%%%%%%%%%%%%%%%%
%%%%%%%%%%%%%%%%%%%%%%%%%%%%%%%%%%%%%
\section{Gerstenhaber brackets on differential graded coalgebras}
\label{sec:gb}

In this section, we summarize some results of \cite{NW1} 
and develop additional techniques for computing Gerstenhaber brackets in the modular setting.
Contrast with~\cite{NW2,SW-G-bracket}, where the characteristic 
of the underlying field was~0. 

%%%%%%%%%%%%%%%%%%%%%%%%%%%
\subsection*{Resolutions as differential graded coalgebras}
Consider a $k$-algebra $A$ and a projective resolution $P$ of $A$
as an $A$-bimodule:
$$
\ldots \rightarrow P_2 \rightarrow P_1 \rightarrow P_0 \rightarrow 0\ .
$$
The resolution $P$ is a {\em differential graded coalgebra}
when $P=\oplus_i P_i$ has a coalgebra structure
compatible with its differential $\del_P$.  This means
there is a (degree $0$) chain map $\Delta_P: P\rightarrow P\ot_A P$
lifting the canonical isomorphism $A\stackrel{\sim}{\longrightarrow} A\ot_A A$,  
called a {\em diagonal map},
that is required to be 
$$
\begin{aligned}
&\text{{\em coassociative}},&&  
\hphantom{x}
\hspace{-3ex}
\text{ i.e., }
(\Delta_P\ot 1)\Delta_P = (1\ot \Delta_P)\Delta_P
\text{ as maps } \ P \rightarrow P\ot_A P \ot_A P ,
\text{ and }
\\
&\text{{\em counital}},&&  
\hphantom{x}
\hspace{-3ex}
\text{ i.e., }
(\mu_P\ot 1_P)\Delta_P = 1_P = (1_P\ot\mu_P)\Delta_P\, 
\text{ as maps } P \rightarrow P ,
\end{aligned}
$$
where 
$\mu_P:P_0\rightarrow A$ is augmentation of the complex 
(with $\mu_P$ zero on $P_i$ for $i\geq 1$). 
Throughout, we define
$\mu_P\ot 1_P: P\ot P\rightarrow P$
as the map $p\ot p'\mapsto \mu_P(p)\cdot p'$
(and similarly for $1_P\ot \mu_P$).
Recall that the differential on $P_n\ot_A P_m$ is just
$\del_P\ot 1_P + (-1)^n 1_P \ot \del_P$.

%%%%%%%%%%%%%%%%%%%%%%
\subsection*{Homotopy from right
to left}
We may map 
the complex $P\ot_A P$ 
to the complex $P$ using either
$\mu_P\ot 1_P$ or $1_P \ot \mu_P$.
When $P$ is a differential graded coalgebra,
these mappings
are chain homotopic by~\cite[Lemma~3.2.1]{NW1}.
(The hypotheses there are slightly stronger,
but the same proof works under our hypotheses here.)
Thus there exists a chain homotopy
from $\mu_P\ot 1_P$ to $1_P \ot \mu_P$, i.e.,
a map $\phi_P: P\ot _A P\rightarrow P$ 
with $P_m\ot_A P_n \rightarrow P_{m+n+1}$
satisfying
\begin{equation}
\label{contractinghomotopy}
     \del_P \, \phi_P + \phi_P\, \del_{P\ot_A P} = \mu_P\ot 1_P - 1_P\ot \mu_P \, .
\end{equation}
%%%%%%%%%%%%%%%%%%%%%%%%%%
\begin{example}{\em 
The bar resolution $B$ of the algebra $A$
is a differential graded coalgebra.
Indeed, for $B_n=A\ot A^{\ot n}\ot A$, a diagonal map
$\Delta_B: B\rightarrow B \otimes_A B$ is defined by 
\begin{equation}
\label{DeltaBar}
\Delta_B(a_0\ot\cdots\ot a_{n+1})=
\sum_{j=0}^n (a_0\ot\cdots\ot a_j \ot 1)\ot _A (1\ot a_{j+1}\ot\cdots
\ot a_{n+1}) 
\end{equation}
for $a_0,\ldots, a_{n+1}$ in $A$.
This map is coassociative and counital.
One choice of homotopy $\phi_B: B\ot_A B \rightarrow B$ 
from $\mu_B\ot 1_B$ to $1_B \ot \mu_B$ is defined by
\begin{equation}\label{eqn:phi-bar}
\begin{aligned}
   \phi_B & \big((a_0\ot\cdots \ot a _{p-1}\ot a_p )\ot_A (a_p'\ot a_{p+1}\ot\cdots
      \ot a_{n+1})\big)
      \\
      & \quad =
  (-1)^{p-1}\,  a_0\ot\cdots\ot a_{p-1}\ot a_p a_p'\ot a_{p+1}\ot\cdots\ot a_{n+1}
\quad\text{  
  for all } a_i, a_p'\in A\, . 
\end{aligned}
\end{equation}
Koszul resolutions of Koszul algebras are also 
differential graded coalgebras~\cite{BGSS}.
The Koszul resolution $P$ of a Koszul algebra embeds into the bar resolution,
however the above map $\phi_B$ does not preserve the image.
Instead, a homotopy $\phi_P$ may be found directly in this case;
see~\cite[\S4]{NW1},~\cite[\S3.2]{NW2}, or~\cite[\S4]{GNW} for some examples. 
}
\end{example}

%%%%%%%%%%%%%%%%%%%%%%%%%%%%%%%%%
\subsection*{Definition of the
Gerstenhaber bracket}
The Gerstenhaber bracket for $A$
is defined on cochains on the bar resolution
$B$ of $A$.
Identify each space of cochains
$\Hom_{A^e}(B_n,A)$ with $\Hom_k(A^{\ot n}, A)$ via the canonical isomorphism. 
Then the Gerstenhaber bracket  
$$
[\ ,\ ]:
\Hom_{k}(A^{\ot n},A) \times \Hom_{k}(A^{\ot m},A) 
\rightarrow \Hom_{k}(A^{\ot(n+m-1)},A) \, 
$$
on cochains
is defined by
$$
[f, f' ] \ =\ f \circ f' - (-1)^{(n-1)(m-1)} f' \circ f
$$
where, for $a_i$ in $A$,
the circle product
$(f  \circ f')(a_1 \ot\cdots\ot a_{n+m-1})$ is
$$
\begin{aligned}
\sum_{i=1}^n
(-1)^{(m-1)(i-1)}
\, f \Big(a_1 \ot \cdots\ot a_{i-1} \ot
 f'(a_i \ot \cdots \ot a_{i+m-1}
) \ot a_{i+m} \ot\cdots\ot a_{n+m-1}\Big)\, .
\end{aligned}
$$

%%%%%%%%%%%%%%%%%%%%%%%%%%%%%%%%%
\subsection*{Gerstenhaber brackets on
differential graded coalgebras}
Although the Gerstenhaber bracket is defined using the bar resolution,
we seek descriptions in terms of more convenient
resolutions used to compute Hochschild cohomology.
Suppose $P$ is a projective
resolution of $A$ with a differential graded coalgebra structure.
The Gerstenhaber bracket 
can be defined directly at the chain level on $P$
using~\cite[Theorem~3.2.5]{NW1};
we recall how a homotopy $\phi_P$ (see~(\ref{contractinghomotopy}))
gives the bracket explicitly.

Extend any cochain
$f\in \Hom_{A^e}(P_n,A)$ to all of $P$ by defining
$f\equiv 0$ on $P_m$ with $m\neq n$.
For
$f\in \Hom_{A^e}(P_n,A)$ and $f'\in\Hom_{A^e}(P_m,A)$, 
define
\begin{equation}\label{eqn:commutator}
    [f,f']_P = f\circ_P f' - (-1)^{(n-1)(m-1)} f'\circ_P f
\end{equation}
where $f\circ_P f'$ (similarly $f'\circ_P f$) is the composition
\begin{equation}\label{eqn:circ}
    f \circ_P f':
    P\stackrel{\Delta^{(2)}_P}{\relbar\joinrel\relbar\joinrel\relbar\joinrel
    \longrightarrow} P\ot_A P\ot_A P\stackrel{1_P\ot f'\ot 1_P}
   {\relbar\joinrel\relbar\joinrel\relbar\joinrel\relbar\joinrel\longrightarrow} P\ot _A P
   \stackrel{\phi_P}{\longrightarrow} P \stackrel{f}{\longrightarrow} A .
\end{equation}
Here, $\Delta^{(2)}_P = (1_P\ot \Delta_P)\Delta_P = (\Delta_P\ot 1_P)
\Delta_P$
and $1_P\ot f' \ot 1_P$ has signs attached so that 
\begin{equation}
\label{signsattached}
(1_P\ot f'\ot 1_P)(x\ot y\ot z) = (-1)^{i m} x\ot f'(y)\ot z
\end{equation}
for $x\in P_i$, $y, z\in P$. 
Then~\cite[Theorem~3.2.5]{NW1} implies that
the Gerstenhaber bracket $[\ , \ ]$
of any elements in cohomology
is given at the cochain level on $P$ by the map $[\ ,\ ]_P$
on cocycles.
(Note that~\cite[Theorem~3.2.5]{NW1} has slightly stronger hypotheses,
but the proof indeed holds for any resolution $P$ with the structure
of a differential graded coalgebra.)

%%%%%%%%%%%%%%%%%%%%%%%%%%%%
%%%%%%%%%%%%%%%%%%%%%%%%%%%%
%%%%%%%%%%%%%%%%%%%%%%%%%%%%
\section{Twisted product resolution
as a differential graded coalgebra}
We show in this section that a twisted product resolution $X$ 
of $S\rtimes G$ constructed from two differential graded coalgebras
$C$ and $D$ is again a 
differential graded coalgebra.
We then give the Gerstenhaber bracket for $X$ in terms of 
the maps describing the Gerstenhaber brackets of $C$ and $D$ individually.

Throughout this section, we fix 
\begin{itemize}
\item
a differential graded coalgebra bimodule resolution $(C, \Delta_C, \mu_C)$ of $G$ and
\item
a differential graded coalgebra bimodule resolution $(D, \Delta_D, \mu_D)$ of $S$, 
producing
\item 
a twisted product resolution $X=C \ot^G D$ 
of $A=S\rtimes G$.
\end{itemize}
We assume that $C$ is $G$-graded (as in~(\ref{G-grading}))
with $\Delta_C,\mu_C$ preserving the grading
and also that
$D$ carries a $G$-action (as in~(\ref{D-compatible}))
with $\Delta_D, \mu_D$ both $kG$-module homomorphisms.
This is the case, for example, if $C$ is the bar (or reduced bar) resolution of $kG$
and $D$ is the Koszul resolution of a Koszul algebra
(see~\cite[Proposition 2.20(ii)]{twisted}).

%%%%%%%%%%%%%%%%%%%%%%%%%%%%
\subsection*{Twisted comultiplication}
In the next lemmas, 
we use diagonal maps for $C$ and $D$ to produce
a diagonal map 
$\Delta_X : X \rightarrow X\ot_A X $.

\begin{lemma}
\label{lem:twisting}
Define a twisting map
$\tau: C\ot  D\rightarrow D\ot  C$
by  
\begin{equation}\label{eqn:tau}
  \tau_{i,j}(c\ot d) = (-1)^{ij} ({}^gd)\ot c
\quad\text{ for all } c\in (C_i)_g \text{ and } d\in D_j\, .
\end{equation}
Then $\tau$ extends to a well-defined chain map 
$$1\ot\tau\ot 1:
(C\ot_{kG}C) \ot (D\ot _S D) \longrightarrow
(C\ot^G D) \ot_{S\rtimes G} (C\ot^G D)\, .$$
\end{lemma}
\begin{proof}
Consider the map
\[
  C\ot C\ot D\ot D \stackrel{1\ot \tau\ot 1}{\relbar\joinrel\relbar\joinrel\relbar\joinrel
\longrightarrow} C\ot D\ot C\ot D \longrightarrow
(C\ot^G D)\ot_{S\rtimes G} (C\ot^G D),
\]
where the latter map is the canonical surjection. 
Calculations show that the composition of these two maps is
$kG$-middle linear in the first two arguments
and $S$-middle linear in the last two arguments, and so
it induces a well-defined map as claimed. 
A calculation shows that it is a chain map. 
\end{proof}

%%%%%%%%%%%%%%%%%%%%%%%%%%%%%%%%%
\begin{lemma}
\label{ComultiplicationOnX}
Let $X=C\ot^G D$ be a twisted product resolution of $S\rtimes G$ 
for differential graded coalgebras 
$C$ and $D$ resolving $kG$ and $S$, respectively, as above.  Then $X$
is a differential graded coalgebra as well with comultiplication
$\Delta_X : X\rightarrow X\ot _{A} X$ given by
$$
\Delta_X=(1\ot \tau\ot 1) (\Delta_C\ot \Delta_D)\, .
$$ 
\end{lemma}
\begin{proof}
We first check that $\Delta_X$ is coassociative
using the fact that
$\Delta_C$ and $\Delta_D$
are each coassociative.  We use
the $G$-grading on $C$
and the compatible $G$-action on $D$: 
\begin{small}
\[
\begin{aligned}
& (\Delta_X\ot 1_X)\Delta_X \\
&= \big((1\ot\tau\ot 1)(\Delta_C\ot\Delta_D)\ot 1\ot 1\big) (1\ot\tau\ot 1)
   (\Delta_C\ot \Delta_D) \\
&= (1\ot \tau\ot 1\ot 1\ot 1) (\Delta_C\ot \Delta_D\ot 1\ot 1)
   (1\ot \tau\ot 1) (\Delta_C\ot \Delta_D) \\
&= (1\ot\tau\ot 1\ot 1\ot 1) (1\ot 1\ot 1 \ot \tau\ot 1)
   (1\ot 1\ot \tau\ot 1\ot 1) (\Delta_C\ot 1 \ot\Delta_D\ot 1)
   (\Delta_C\ot\Delta_D) \\
&= (1\ot \tau\ot 1\ot 1\ot 1) (1\ot 1\ot 1\ot \tau\ot 1) 
  (1\ot 1\ot \tau\ot 1\ot 1) (1\ot \Delta_C\ot 1\ot \Delta_D)
   (\Delta_C\ot \Delta_D)\\
  &= (1\ot 1\ot 1\ot \tau\ot 1) (1\ot \tau\ot 1\ot 1\ot 1)
  (1\ot 1\ot\tau\ot 1\ot 1) (1\ot\Delta_C\ot 1\ot\Delta_D)
  (\Delta_C\ot \Delta_D) \\
&= (1\ot 1\ot 1\ot \tau\ot 1)(1\ot 1\ot \Delta_C\ot \Delta_D)
   (1\ot\tau\ot 1) (\Delta_C\ot \Delta_D) \\
 &= (1_X\ot \Delta_X) \Delta_X . 
\end{aligned}
\]
\end{small}We 
next verify that $\Delta_X$ is counital
using the fact that $\Delta_C$ and $\Delta_D$ are each counital.
We use the extra
assumption that $\mu_C$ preserves the $G$-grading
and $\mu_D$ is a $kG$-module homomorphism as well
as the definition of the $S\rtimes G$-bimodule structure
on $C\ot^G D$: 
\begin{small}
\[
\begin{aligned}
(\mu_X\ot 1_X)\Delta_X
&= (\mu_C\ot \mu_D\ot 1\ot 1)(1\ot \tau\ot 1) (\Delta_C\ot \Delta_D)\\
&= (\mu_C\ot 1\ot\mu_D\ot 1) (\Delta_C\ot \Delta_D) 
= ((\mu_C\ot 1)\Delta_C) \ot ((\mu_D\ot 1) \Delta_D) \\
&= 1\ot 1 \ \ = \ \ 1_X ,
\end{aligned}
\]
\end{small}and, 
similarly, $(1_X\ot \mu_X)\Delta_X = 1_X$. 

We now need only check that $\Delta_X$ is a chain map, 
i.e.,~$\Delta_X\, \del = (\del\ot 1 + 1\ot \del)\Delta_X$,
for $\del$ the differential on $X$.
This follows from the fact that
$\tau, \Delta_C,\Delta_D$ are all chain maps. 
\end{proof}

%%%%%%%%%
\begin{remark}{\em
One may check that the map $1\ot \tau\ot 1$
of~(\ref{eqn:tau}) interpolates between
the maps of the form $\mu\ot 1 - 1\ot\mu$ for the various complexes,
that is,
\begin{equation}\label{eqn:mux}
\mu_X\ot 1_X - 1_X\ot \mu_X = (\mu_C\ot 1_C\ot \mu_D\ot 1_D
    - 1_C\ot\mu_C\ot 1_D\ot\mu_D) (1_C\ot\tau^{-1}\ot 1_D) .
\end{equation}
}
\end{remark}

We now give a theorem describing a homotopy from
$\mu_X\ot 1_X$ to $1_X\ot\mu_X$ concretely in terms of
homotopies from 
$\mu_C\ot 1_C$ to $1_C\ot \mu_C$ and
from $\mu_D\ot 1_D$ to $1_D\ot \mu_D$
by adapting~\cite[Lemmas~3.3, 3.4, and 3.5]{GNW}
to our setting.

%%%%%%%%%%%%%%%%%%%%%%%%%%%%%%%
\begin{thm}\label{thm:phi}
Let $X=C \ot^G D$ as above
with  homotopies $\phi_C$ 
from
$\mu_C\ot 1_C$ to $1_C\ot \mu_C$
and $\phi_D$ from $\mu_D\ot 1_D$ to $1_D\ot \mu_D$. 
Define $\phi_X: X\ot _{A}X\rightarrow X$ by
\[
     \phi_X=  \big(\phi_{C} \ot \mu_D\ot 1_D + 
       \ep_C(1_C\ot\mu_C)\ot \phi_{D}\big) (1_C\ot \tau^{-1}\ot 1_D)
\]
for $\epsilon_C: C \rightarrow C$
defined by $c\mapsto (-1)^{|c|}$ for homogeneous $c$.
Then $\phi_X$ is a homotopy from $\mu_X\ot 1_X$ to $1_X\ot\mu_X$. 
\end{thm}

\begin{proof}
Let
$\phi_X':C\ot C\ot D\ot D\rightarrow C\ot D$
be the map $\phi_C\ot\mu_D\ot 1+
\ep_C(1\ot \mu_C)\ot \phi_D$
so that
$\phi_X=\phi_X'(1\ot \tau^{-1}\ot 1)$.
Then on $(C\ot C) \ot (D\ot D)$,
\begin{align}
\label{first}
  \del_X\, \phi_X\ (1\ot \tau\ot 1)
  &=\ \del_X\, \phi_X'
  \nonumber\\
  &=\ (\del_C\ot 1 + \ep_C \ot \del_D)
  \big(\phi_C\ot \mu_D\ot 1
  +  \ep_C(1\ot\mu_C)\ot \phi_D\big)
\nonumber\\
  &=\ \del_C\phi_C\ot \mu_D\ot 1
  - \phi_C(\ep_C\ot \ep_C)\ot
   \del_D(\mu_D\ot 1)
   \nonumber\\
   &\quad\quad
   - \ep_C\del_C(1\ot \mu_C)\ot \phi_D
   + 1 \ot \mu_C\ot \del_D\phi_D\, ,
\end{align}
and, since $1\ot \tau \ot 1$
is a chain map from $(C\ot C)\ot (D\ot D)$ to $X\ot_A X$, 
\begin{align} 
  \phi_X\, \del_{X\ot X} \, (1\ot\tau\ot 1)
  &=
  \phi_X\, (1\ot\tau\ot 1)\, \del_{(C\ot C)\ot (D \ot D)}
  \nonumber
  = 
  \phi_X'\ \del_{(C\ot C)\ot (D \ot D)}
  \nonumber\\
  &=\phi_X'\, \big(\del_{C\ot C}\ot 1_{D\ot D}+(\ep_C\ot \ep_C)\ot \del_{D\ot D}\big)
  \nonumber\\
  &= \phi_C\del_{C\ot C}\ot \mu_D\ot 1
  + \phi_C(\ep_C\ot\ep_C)\ot (\mu_D\ot 1)\del_{D\ot D}
  \nonumber\\
\label{second}
  &\quad\quad
  +\ep_C(1\ot\mu_C)\del_{C\ot C}\ot \phi_D
  +(1\ot\mu_C)\ot\phi_D\del_{D\ot D}.
\end{align}
Here we used the fact that $\ep_C\phi_C=-\phi_C(\ep_C\ot\ep_C)$,
$\ep_C(1\ot\mu_C)(\ep_C\ot \ep_C)=1\ot \mu_C$, 
and $\del_C\ep_C=-\ep_C\del_C$.
The second term of~(\ref{first}) cancels with
the second term of~(\ref{second}) as $\mu_D\ot 1$ is a chain map;
likewise,
the third terms cancel as $\mu_C\ot 1$ is a chain map.
Hence
$$
\begin{aligned}
  (\del_X\, \phi_X\ +  & \phi_X\, \del_{X\ot X} ) 
  (1\ot \tau\ot 1)\\
&= (\del\phi_C + \phi_C \del)\ot \mu_D\ot 1 + 
1\ot\mu_C\ot (\del\phi_D +\phi_D \del)\\
     & = (\mu_C\ot 1 - 1\ot\mu_C)\ot \mu_D\ot 1 + 
     1\ot\mu_C\ot (\mu_D\ot 1- 1\ot\mu_D)
     \\
     & = \mu_C\ot 1\ot \mu_D\ot 1 - 1\ot\mu_C\ot 1\ot\mu_D,
\end{aligned}
$$
and, by
equation~(\ref{eqn:mux}), 
\[
\begin{aligned}
  \del\phi_X &+ \phi_X \del
     = (\mu_C\ot 1\ot \mu_D\ot 1 - 1\ot\mu_C\ot 1\ot\mu_D)
     (1\ot\tau^{-1}\ot 1)  
     = \mu_X\ot 1 - 1\ot\mu_X .
\end{aligned}
\]
\end{proof} 

%%%%%%%%%%%%%%%%%%%%%%%%%%%%%%%
\subsection*{Gerstenhaber bracket for skew group algebras}

The next theorem gives the Gerstenhaber bracket 
on a twisted product resolution $X$.
Note that
the twisting map $\tau$ in the theorem is from Lemma~\ref{lem:twisting},
the map $1_X\ot f' \ot 1_X$ has signs attached as in~(\ref{signsattached}),
and
$\epsilon_C$ merely adjusts signs, 
$c\mapsto (-1)^{ |c|}$  for homogeneous $c$ in $C$.

\begin{thm}
\label{maintheorem}
Let $X=C\ot^G D$ be a twisted product resolution of $S\rtimes G$ 
for differential graded coalgebras 
$(C, \Delta_C, \mu_C)$ and $(D,\Delta_D, \mu_D)$ resolving $kG$ and $S$, respectively, as above. 
The Gerstenhaber bracket of elements of Hochschild cohomology
represented by cocycles
$f\in \Hom_{A^e}(X_n,A)$ and $f'\in\Hom_{A^e}(X_m,A)$ is
represented by the cocycle
\begin{equation}\label{eqn:commutator2}
    [f,f'] = f\circ_X f' - (-1)^{(n-1)(m-1)} f'\circ_X f\, ,
\end{equation}
where
$f\circ_X f'$
(similarly $f'\circ_X f$) is the composition
\begin{equation}
\label{eqn:circ2}
    X\xrightarrow{\ \ (1_X\ot\Delta_X)(\Delta_X)\ \ }
    X\ot_A X\ot_A X\xrightarrow{\ 1_X\ot f'\ot 1_{X} \ }
   X\ot _A X
   \xrightarrow{\ \phi_X\ }
   X \xrightarrow{\ f\ } A 
\end{equation}
with 
$$
\begin{aligned}
\Delta_X =&\ (1_C\ot \tau\ot 1_D) (\Delta_C\ot \Delta_D)\, ,
\text{ and } \\
\phi_X =&\ \big(\phi_{C} \ot \mu_D\ot 1_D + 
  (1\ot\mu_C)(\ep_C\ot 1)\ot \phi_{D}\big) (1\ot \tau^{-1}\ot 1)\, .
\end{aligned}
$$

\end{thm}
\begin{proof}
We combine
Lemmas~\ref{lem:twisting}, Lemma~\ref{ComultiplicationOnX}, and Theorem~\ref{thm:phi} with~(\ref{eqn:circ}) and~(\ref{eqn:commutator}).
\end{proof}

\begin{example}{\em
In case $S=S(V)\cong k[x_1, \ldots, x_n]$, 
the symmetric algebra on a finite dimensional vector space $V$,
we take $D$ to be the Koszul resolution for which a choice
of $\phi_D$ has been made in~\cite[\S4]{NW1} (see also~\cite[\S3.2]{NW2}).
We may take $C$ to be the bar or reduced bar resolution of $kG$ for some applications,
with homotopy $\phi_C$ as defined by 
equation~(\ref{eqn:phi-bar}). 
}
\end{example}

%%%%%%%%%%%%%%%%%%%%%%%%%%%%%%%%%%%%%%%
%%%%%%%%%%%%%%%%%%%%%%%%%%%%%%%%%%%%%%%
%%%%%%%%%%%%%%%%%%%%%%%%%%%%%%%%%%%%%%%
\section{A small transvection group example}\label{sec:example}
We end by demonstrating how to use a twisted product resolution
to compute Gerstenhaber brackets explicitly via Theorem~\ref{maintheorem}. 
We also see how computation of explicit brackets
can shed light on questions
in deformation theory (see~\cite{ueber}).
We illustrate with the
prototype example of
a graded Hecke algebra (or rational Cherednik algebra)
in positive characteristic
(see~\cite{Norton} and~\cite{ueber}).
In the nonmodular setting, these algebras
have parameters supported only on the identity
group element and on
bireflections; in the modular setting,
parameters can also be supported on reflections.
All reflections in a finite linear 
group $G$ acting in the modular setting
are either diagonalizable or 
act as in this example.
We include some explicit details to illustrate
how to evaluate the maps in Theorem~\ref{maintheorem}
concretely.
We find both a nonzero and a zero Gerstenhaber bracket.

\subsection{Group action and twisted product resolution}
Say $\text{char}(k)=p>0$ and consider the cyclic group $G\simeq\mathbb{Z}/p\mathbb{Z}$ 
acting on $V=k^2$ with basis $v,w$ 
generated by$$
g=\left(\begin{smallmatrix} 1 & 1\\0 & 1 \end{smallmatrix}\right),
\quad\text{ so that }\quad
    {}^g v = v \ \mbox{ and }  \ {}^gw = v+w .
$$
We work in the twisted product resolution 
$X=C\ot^G D$ of $S(V)\rtimes G$ obtained
from twisting the reduced bar resolution $C$
of $kG$ with the Koszul resolution $D$ of $S(V)$:
\[ 
X_n = \bigoplus_{i+j=n} X_{i,j}
\quad\text{for}\quad
   X_{i,j} = kG\ot (\overline{kG})^{\ot i}\ot kG
   \ot S(V)\ot \Wedge^j V \ot S(V).
\]
Here, $C_n=kG\ot (\overline{kG})^{\ot n} \ot kG$
with $\overline{kG}=kG/k1_G$ and
$D_n = S(V)\ot \Wedge^n V\ot S(V)$.
Then $C$ and $D$ satisfy the conditions
specified in Section~\ref{sec:gb}, and
Theorem~\ref{maintheorem} applies.

\subsection{Cochains}
Consider cochains on the resolution $X$: 
$$
\begin{aligned}
&\kappa\in\Hom_{(S(V)\rtimes G)^e}(X_{0,2}, S(V)\rtimes G), \\
&\lambda\in\Hom_{(S(V)\rtimes G)^e} (X_{1,1}, S(V)\rtimes G), 
\mbox{ and } \\
&\delta\in\Hom_{(S(V)\rtimes G)^e}(X_{0,1}, S(V)\rtimes G)
\end{aligned}
$$
defined by (with subscripts on the tensor signs suppressed for brevity)
\begin{small}
\begin{eqnarray*}
    \lambda\big((1_G\ot g^i\ot 1_G)\ot (1_S\ot v\ot 1_S)\big) &=& 0 , \\ 
   \lambda\big((1_G\ot g^i\ot 1_G)\ot (1_S\ot w\ot 1_S)\big) &=& \ ig^{i-1}, \\
   \kappa \big((1_G\ot 1_G)\ot (1_S\ot v\wedge w\ot 1_S)\big) &=& g,  \\ 
   \delta\big((1_G\ot 1_G)\ot (1_S\ot v\ot 1_S)\big) &=& v, \\
   \delta\big((1_G\ot 1_G) \ot (1_S\ot w\ot 1_S)\big) &=& 0 
\end{eqnarray*}
\end{small}for
$0\leq i\leq p-1$, with all other values determined by these.
One can check directly that $\lambda$ and $\kappa$
are 2-cocycles
and that $\delta$ is a $1$-cocycle for $X$.
We will show that 
$$
[\delta,\kappa]\neq 0\quad\text{ and }\quad
[\lambda,\lambda]=[\lambda,\kappa]=0\, .
$$

\subsection*{The diagonal maps}
We give some values of the diagonal maps at play in finding 
the Gerstenhaber brackets.
The diagonal map $\Delta_C$ on the reduced bar resolution of $kG$ is 
deduced from~(\ref{DeltaBar}).  For example, after identifying $g^i$ 
with its image in $\overline{kG}$,
\begin{small}
$$
\begin{aligned}
\Delta_C(1_G \ot g^i \ot 1_G) =&\ 
(1_G\ot 1_G) \ot_{kG} (1_G\ot g^i \ot 1_G) +
(1_G\ot g^i \ot 1_G)\ot_{kG} (1_G\ot 1_G),\ \text{ and }
\\
\Delta_C(1_G\ot 1_G) =&\ (1_G\ot 1_G) \ot_{kG} (1_G\ot 1_G)\, .
\end{aligned}
$$
\end{small}The 
diagonal map $\Delta_D$ is found
from embedding the Koszul into the bar resolution and then using~(\ref{DeltaBar}).
For example, we identify $v\wedge w$ with $v\ot w - w\ot v$ and observe
that
\begin{small}
\begin{eqnarray*}
  \Delta_D (1_S\ot v\wedge w\ot 1_S) 
  & = & 
  (1_S\ot 1_S)\ot_S (1_S\ot v\wedge w\ot 1_S)\\
   && + (1_S\ot v\ot 1_S)\ot_S (1\ot w\ot 1) - (1_S\ot w\ot 1_S)
   \ot_S (1_S\ot v \ot 1_S) 
      \\
  && + (1_S\ot v\wedge w\ot 1_S)\ot_S (1_S\ot 1_S) \, ,\quad\text{ and }
\\
  \Delta_D (1_S\ot v\ot 1_S) & = & 
  (1_S\ot 1_S)\ot_S (1_S\ot v\ot 1_S)
  + (1_S\ot v\ot 1_S)\ot_S (1_S \ot 1_S) \, .
  \end{eqnarray*}
  \end{small}

\vspace{-1ex}
\subsection*{Homotopies}
Let $\phi_C: C\ot_{kG} C \rightarrow C$ be the homotopy
from $\mu_C\ot 1$ to $1\ot \mu_C$
from~(\ref{eqn:phi-bar}).
We choose the homotopy
$\phi_D: D\ot_{S}D\rightarrow D$ 
from $\mu_D\ot 1$ to $1\ot \mu_D$
given in~\cite[Definition~4.1.3]{NW1}
and record
a few values here for later use:
\begin{small}
  $$
\begin{aligned}
\phi_D((1\ot w\ot 1)\ot_S (v\ot 1))=&\ 1\ot v\wedge w\ot 1,
\quad
&\phi_D((1\ot v)\ot_S (1\ot w\ot 1))=\ 0,
\\
\phi_D((1\ot 1)\ot_S (1\ot v\ot 1)) =&\ 0,
\quad
&\ \ \ \phi_D((1\ot v\ot 1)\ot_S (1\ot 1))\ =\ 0\, .
\end{aligned}
$$
\end{small}

\vspace{-1ex}
\subsection*{Nonzero bracket}
We use Theorem~\ref{maintheorem} to show explicitly that $[\delta , \kappa ] =  \kappa$.
First note that $[\delta,\kappa]$ is zero on
all components of $X$ except
possibly $X_{0,2}$.
We consider the composition~(\ref{eqn:circ2}) with $f'=\delta$ and
$f=\kappa$ to find $\kappa\circ_X\delta$.
As a first step, we apply 
the map $(\Delta_X\ot 1_X)\Delta_X$ to the element 
$(1_G\ot 1_G)\ot (1_S\ot v\wedge w\ot 1_S)$
of $X_{0,2}$, where, recall
$$
\Delta_X =(1_C\ot \tau\ot 1_D) (\Delta_C\ot \Delta_D)\, .
$$ 
Direct calculation confirms that
\begin{small}
\begin{eqnarray*}
&& \hspace{-15ex} 
(1\ot 1)\ot (1\ot v\wedge w\ot 1) 
\xmapsto{\ (\Delta_X\ot 1)\Delta_X \ }\\
&& (1\ot 1)\ot (1\ot 1)\ot (1\ot 1)\ot (1\ot 1)\ot (1\ot 1)
      \ot (1\ot v\wedge w\ot 1)\\
  && +(1\ot 1)\ot (1\ot 1)\ot (1\ot 1)\ot (1\ot v\ot 1)\ot (1\ot 1)\ot (1\ot w\ot 1)\\
  && +(1\ot 1)\ot (1\ot v\ot 1)\ot (1\ot 1)\ot (1\ot 1)\ot (1\ot 1)\ot (1\ot w\ot 1)\\
  && -(1\ot 1)\ot (1\ot 1)\ot (1\ot 1)\ot (1\ot w\ot 1)\ot (1\ot 1)\ot (1\ot v\ot 1)\\
  && - (1\ot 1)\ot (1\ot w\ot 1)\ot (1\ot 1)\ot (1\ot 1)\ot (1\ot 1)\ot (1\ot v\ot 1)\\
  && + (1\ot 1)\ot (1\ot 1)\ot (1\ot 1)\ot (1\ot v\wedge w\ot 1) \ot (1\ot 1)\ot (1\ot 1)\\
  && + (1\ot 1)\ot (1\ot v\ot 1)\ot (1\ot 1)\ot (1\ot w\ot 1) \ot (1\ot 1)\ot (1\ot 1)\\
  && - (1\ot 1)\ot (1\ot w\ot 1)\ot (1\ot 1)\ot (1\ot v\ot 1)\ot (1\ot 1)\ot (1\ot 1)\\
  && + (1\ot 1)\ot (1\ot v\wedge w\ot 1)\ot (1\ot 1)\ot (1\ot 1)\ot (1\ot 1)\ot (1\ot 1)
\end{eqnarray*}
\end{small}\@as an element of $X\ot_A X\ot_A X$.
We have suppressed all subscripts for brevity; for example, 
the second summand may be written
\begin{small}
$$
  \big((1_G\ot_{kG} 1_G)\ot (1_S\ot_S 1_S)\big)\ot_A 
  \big((1_G\ot_{kG} 1_G)\ot (1_S\ot_S v\ot_S 1_S)\big)\ot_A 
  \big((1_G\ot_{kG} 1_G)
    \ot (1_S\ot_S w\ot_S 1_S)\big)\, .
$$
\end{small}We
next apply the map
$1_X\ot \delta\ot 1_X$;
it is nonzero on exactly two summands, the second and the penultimate, and we obtain
(with the tensor products over $A$ indicated here)
\begin{small}
\[
  \big((1_G\ot 1_G)\ot (1_S\ot v)\big)
  \ot_A 
  \big((1_G\ot 1_G)\ot (1_S\ot w\ot 1_S)\big)
   - 
   \big((1_G\ot 1_G)\ot (1_S\ot w\ot v)\big)
   \ot_A
   \big( (1_G\ot 1_G)\ot (1_S\ot 1_S)\big) .
\]
\end{small}To apply $\phi_X$ next, we first 
rearrange terms with
$1_G\ot \tau^{-1}\ot 1_S$, producing
\begin{small}
\[
   (1_G\ot 1_G)\ot (1_G\ot 1_G)\ot (1_S\ot v)\ot (1_S\ot w\ot 1_S) 
    - (1_G\ot 1_G)\ot (1_G\ot 1_G)\ot (1_S\ot w\ot v)\ot (1_S\ot 1_S), 
\]
\end{small}and 
then apply the map $\phi_C\ot \mu_D\ot 1_D + 1_C\ot\mu_C\ot \phi_D$
to obtain
\begin{small}
\[
   (1_G\ot 1_G\ot 1_G)\ot (v\ot w\ot 1_S) 
   - (1_G\ot 1_G)\ot (1_S\ot v\wedge w\ot 1_S).
   \]\end{small}Lastly,
   we apply $\kappa$  as the last step of~(\ref{eqn:circ2})
and obtain $0$ from 
the first term and $-g$ from the second.
Thus 
\begin{small}
\[
    (\kappa\circ\delta)\big((1_G\ot 1_G)\ot (1_S\ot v\wedge w\ot 1_S)\big) 
    = -g 
    = \kappa\big( (1_G\ot 1_G) \ot (1_S\ot v\wedge w \ot 1_S) \big)
    \,
\]
\end{small}and $\kappa\circ_X\delta = - \kappa$.
We inspect the above calculation with an eye toward switching the order
of $\kappa$ and $\delta$ and deduce that  $\delta\circ_X\kappa = 0$.
We conclude, as claimed,
\[
   [\delta,\kappa] = \delta\circ_X\kappa - \kappa\circ_X\delta = \kappa \, .
\]

%%%%%%%%%%%%%%%%%%%%%%%%%%
\subsection*{Zero brackets}
We now use Theorem~\ref{maintheorem} to show that  $[\lambda,f]=0$ 
when $f$ is $\lambda$ or $\kappa$.
We evaluate
composition~(\ref{eqn:circ2}) 
on $X_{1,2}$ with $f'=\lambda$.
Other calculations are similar.
We first apply 
$
\Delta_X =(1_G\ot \tau\ot 1_S) (\Delta_C\ot \Delta_D)\, 
$
to sample input in $X_{1,2}$, noting that
$ {}^{g^i} w = iv+w$ (with  subscripts suppressed again):
\begin{small}
\begin{eqnarray*}
&&\hspace{-25ex}(1\ot g^i\ot 1)\ot (1\ot v\wedge w\ot 1)\\
&\xmapsto{\Delta_C\ot \Delta_D}
& (1\ot 1) \ot (1\ot g^i\ot 1)\ot (1\ot 1)\ot (1\ot v\wedge w\ot 1)\\
 &&+(1\ot 1)\ot (1\ot g^i\ot 1)\ot (1\ot v\ot 1)\ot (1\ot w\ot 1)\\
&&- (1\ot 1)\ot (1\ot g^i\ot 1)\ot (1\ot w\ot 1)\ot (1\ot v\ot 1)\\
 && + (1\ot 1)\ot (1\ot g^i\ot 1)\ot (1\ot v\wedge w\ot 1)\ot (1\ot 1)\\
&&+(1\ot g^i\ot 1)\ot (1\ot 1)\ot (1\ot 1) \ot (1\ot v\wedge w\ot 1)\\
 &&+(1\ot g^i\ot 1)\ot (1\ot 1)\ot (1\ot v\ot 1)\ot (1\ot w\ot 1)\\
  &&- (1\ot g^i\ot 1)\ot (1\ot 1)\ot (1\ot w\ot 1)\ot (1\ot v\ot 1)\\
 &&+ (1\ot g^i\ot 1)\ot (1\ot 1)\ot (1\ot v\wedge w\ot 1)\ot (1\ot 1)\, 
\end{eqnarray*}
\end{small}
\vspace{-4ex}
\begin{small}
\begin{eqnarray*}
\hspace{5ex}
&\xmapsto{1\ot \tau\ot 1}&
 (1\ot 1)\ot (1\ot 1)\ot (1\ot g^i\ot 1)\ot (1\ot v\wedge w\ot 1)\\
&& - (1\ot 1)\ot (1\ot v\ot 1)\ot (1\ot g^i\ot 1)\ot (1\ot w\ot 1)\\
&&+(1\ot 1)\ot (1\ot (iv+w)\ot 1) \ot (1\ot g^i\ot 1)\ot (1\ot v\ot 1)\\
&&+ (1\ot 1)\ot (1\ot v\wedge w\ot 1)\ot (1\ot g^i\ot 1)\ot (1\ot 1)\\
&&+ (1\ot g^i\ot 1) \ot (1\ot 1)\ot (1\ot 1)\ot (1\ot v\wedge w\ot 1)\\
&&+(1\ot g^i\ot 1)\ot (1\ot v\ot 1)\ot (1\ot 1)\ot (1\ot w\ot 1)\\
 &&-(1\ot g^i\ot 1)\ot (1\ot w\ot 1)\ot (1\ot 1)\ot (1\ot v\ot 1)\\
&&+(1\ot g^i\ot 1)\ot (1\ot v\wedge w\ot 1)\ot (1\ot 1)\ot (1\ot 1)\, ,
\end{eqnarray*}
\end{small}an element of $X\ot_A X$.
Next we apply $\Delta_X\ot 1_X$:
Evaluating ${\Delta_C\ot \Delta_D\ot 1_X}$ on the last expression
yields 27 summands;
the map $(1_G\ot\tau\ot 1_S)\ot 1_X $ 
transforms these to 27 summands in $X\ot_A X\ot_A X$.
A quick check verifies that $1_X\ot \lambda\ot 1_X$ vanishes on all but two summands,
namely
\begin{small}
\[
\begin{aligned}
& - \big((1_G\ot 1_G)\ot (1_S\ot 1_S)\big)
\ot_A 
\big((1_G\ot g^i\ot 1_G)\ot (1_S\ot w\ot 1_S)\big)
\ot_A 
\big((1_G\ot 1_G)\ot (1_S\ot v\ot 1_S)\big),
\\
& -\big((1_G\ot 1_G)\ot (1_S\ot v\ot 1_S)\big)
\ot_A 
\big(1_G\ot g^i\ot 1_G)\ot (1_S\ot w\ot 1_S)\big)
  \ot_A 
  \big((1_G\ot 1_G)\ot (1_S\ot 1_S) \big),
\end{aligned}
\]
\end{small}
and we obtain 
\begin{small}
\[
\begin{aligned}
&&
-\big((1_G\ot 1_G)\ot (1_S\ot 1_S)\big) 
\ot_A 
\big((ig^{i-1}\ot 1_G)\ot (1_S\ot v\ot 1_S)\big)\\
&&
- \big((1_G\ot 1_G)\ot (1_S\ot v\ot 1_S)\big)
\ot_A 
\big((ig^{i-1}\ot 1_G)\ot (1_S\ot 1_S)\big) .
\end{aligned}
\]
\end{small}Applying 
$\phi_X$ followed by $f=\lambda$ 
or $f=\kappa$ gives
$ 0$ as  $w$ does not appear in the input.

\begin{remark}{\em 
The cocycles $\lambda$ and $\kappa$ above were
not chosen randomly.  These cocycles
define a PBW deformation of $S\rtimes G$, and
the zero brackets calculated above predict the PBW property.  
Indeed, in~\cite{ueber}, we considered
PBW deformations 
of $S\rtimes G$ given by analogs
of Lusztig's graded Hecke algebras and symplectic reflection algebras
over fields of positive characteristic.
These algebras $\mathcal{H}_{\lambda,\kappa}$ 
depend on two parameters $\lambda$ and $\kappa$
with $\lambda: kG\otimes V \rightarrow kG$
and $\kappa: V\otimes V \rightarrow kG$.
The Hochschild 2-cocycles above of the same name
$\lambda$ and $\kappa$ are these parameters
converted into cocycles
on the resolution $X$;
see~\cite[Example~2.2]{ueber} and 
also~\cite[Section~5]{AW-EZ}.
A necessary condition for the parameters $\lambda $ and $\kappa$
to define a PBW deformation is that
$$
[\lambda, \lambda]=0 \text{ and }
[\lambda, \kappa]=0
$$
when the cochains $\kappa$ and $\lambda$ they define
are cocycles.  (More generally, we require
that $\lambda$ is a cocycle, $[\lambda, \lambda]=0$,
and $[\lambda, \lambda]=2\del^*\kappa$.)
Thus knowing explicit values for brackets is helpful 
for finding new deformations.
The cocycle $\delta$ above is included merely for
illustration purposes; it provides an example
of a nonzero Gerstenhaber bracket.
}
\end{remark}
%%%%%%%%%%%%%%%%%%%%%%%%%%%%%%%%%%%%%%%%%%%%%%%%%%5

\end{document}